\documentclass[11pt,twoside,reqno]{amsart}
\usepackage[latin1]{inputenc}
\usepackage[OT1]{fontenc}
\usepackage{amsmath}
\usepackage{amsthm}
\usepackage{amssymb}
\usepackage[all]{xy}
\usepackage{bbm}
\usepackage{amscd}
\usepackage{a4wide}
\usepackage{enumitem}
\usepackage{stmaryrd}
\usepackage{mathtools}
\usepackage{csquotes}

\setenumerate[0]{label=(\roman*)}

\DeclareMathOperator{\Stab}{\mathsf{Stab}}

\DeclareMathOperator{\Perf}{\mathsf{Perf}}

\DeclareMathOperator{\id}{\mathsf{id}}

\DeclareMathOperator{\Hom}{\mathsf{Hom}}

\DeclareMathOperator{\End}{\mathsf{End}}

\DeclareMathOperator{\Coh}{\mathsf{Coh}}
\DeclareMathOperator{\QCoh}{\mathsf{QCoh}}

\DeclareMathOperator{\Mod}{\mathsf{Mod}}
\DeclareMathOperator{\dgMod}{\mathsf{dgMod}}
\DeclareMathOperator{\Ac}{\mathsf{Ac}}
\DeclareMathOperator{\rep}{\mathsf{rep}}

\DeclareMathOperator{\fdMod}{\mathsf{mod}}
\DeclareMathOperator{\Char}{char}

\DeclareMathOperator{\Ind}{\mathsf{Ind}}

\DeclareMathOperator{\irr}{\mathsf{irr}}

\DeclareMathOperator{\Aut}{\mathsf{Aut}}

\newcommand{\fk}{{\mathbf k}}

\newcommand{\IE}{\mathbb{E}}
\newcommand{\IF}{\mathbb{F}}

\newcommand{\IR}{\mathbb{R}}
\newcommand{\IU}{\mathbb{U}}
\newcommand{\IZ}{\mathbb{Z}}

\DeclareMathOperator{\Res}{\mathsf{Res}}

\makeatletter
\newcommand{\leqnomode}{\tagsleft@true}
\newcommand{\reqnomode}{\tagsleft@false}
\makeatother

\let\dim\relax
\DeclareMathOperator{\dim}{\mathsf{dim}}

\newcommand{\sym}{\mathfrak S}

\newcommand{\cA}{\mathcal A}
\newcommand{\cB}{\mathcal B}

\newcommand{\cF}{\mathcal F}
\newcommand{\cG}{\mathcal G}

\newcommand{\cE}{\mathcal E}

\newcommand{\cC}{\mathcal C}
\newcommand{\cD}{\mathcal D}

\newcommand{\cT}{\mathcal T}

\newcommand{\eps}{\varepsilon}

\renewcommand{\theta}{\vartheta}
\renewcommand{\rho}{\varrho}
\renewcommand{\phi}{\varphi}
\renewcommand{\_}{\underline{\,\,\,\,}}

\usepackage[usenames,dvipsnames]{xcolor}
\usepackage{hyperref}
\hypersetup{
    colorlinks,
    linkcolor={red!50!black},
    citecolor={blue!50!black},
    urlcolor={blue!80!black}
}
\usepackage{aliascnt} 

\newtheorem{theorem}{Theorem}[section]

\newaliascnt{conjecture}{theorem}

\aliascntresetthe{conjecture}

  \newaliascnt{proposition}{theorem}
  \newtheorem{prop}[proposition]{Proposition}
  \aliascntresetthe{proposition}

  \newaliascnt{lemma}{theorem}
  \newtheorem{lemma}[lemma]{Lemma}
  \aliascntresetthe{lemma}

  \newaliascnt{corollary}{theorem}
  \newtheorem{cor}[corollary]{Corollary}
  \aliascntresetthe{corollary}

\newaliascnt{propdef}{theorem}
  \newtheorem{propdef}[propdef]{Proposition and Definition}
  \aliascntresetthe{propdef}

\theoremstyle{definition}

  \newaliascnt{definition}{theorem}
  \newtheorem{definition}[definition]{Definition}
  \aliascntresetthe{definition}

  \newaliascnt{remark}{theorem}
  \newtheorem{remark}[remark]{Remark}
  \aliascntresetthe{remark}

  \newaliascnt{condition}{theorem}
  
  \aliascntresetthe{condition}

\newaliascnt{convention}{theorem}
 
\aliascntresetthe{convention}

  \newaliascnt{question}{theorem}
  
  \aliascntresetthe{question}

  \newaliascnt{example}{theorem}
  
  \aliascntresetthe{example}

  \newaliascnt{examples}{theorem}
  \newtheorem{examples}[examples]{Examples}
  \aliascntresetthe{examples}

\begin{document}

\title{Endomorphism Algebras of Equivariant Exceptional Collections}
\author[A. Krug]{Andreas Krug}
\author[E. Nikolov]{Erik Nikolov}
\address{
Institut f\"ur algebraische Geometrie,
Gottfried Wilhelm Leibniz Universit\"at Hannover,
Welfengarten 1,
30167 Hannover,
Germany
}
\email{krug@math.uni-hannover.de, nikolov@math.uni-hannover.de}
\date{March 28, 2024}

\begin{abstract}
Given an action of a finite group on a triangulated category with a suitable strong exceptional collection, a construction of Elagin produces an associated strong exceptional collection on the equivariant category.\ We prove that the endomorphism algebra of the induced exceptional collection is the basic reduction of the skew group algebra of the endomorphism algebra of the original exceptional collection.
\end{abstract}
 
\maketitle 
 
\section{Introduction} 

In algebraic geometry, more precisely in the study of quotient singularities and the McKay correspondence, \emph{equivariant derived categories} often play an important role.\ That is, given an action of a finite group on a variety $X$, there is the notion of $G$-equivariant coherent sheaves on $X$.\ They can be seen as generalisations of vector bundles over $X$ carrying a $G$-action which is compatible with the $G$-action on $X$ and are identified with the coherent sheaves on the quotient stack $[X/G]$.\ The $G$-equivariant coherent sheaves form an abelian category $\Coh_G(X)$ and the associated bounded derived category $D^b_G(X):=D^b(\Coh_G(X))$ can be considered.

Elagin gave a method to construct exceptional collections (or sequences) on $D^b_G(X)$ out of suitable exceptional sequences on $D^b(X)$:
\begin{theorem}[{\cite[Thm.\ 2.3]{Ela}}]\label{thm:intro1}
Let $G$ act on a smooth variety $X$ over an algebraically closed field  of characteristic zero or coprime to $|G|$.\ Assume that $D^b(X)$ possesses a full exceptional collection of the form
\begin{equation}\label{eq:originalECintro}
D^b(X)=\bigl\langle E_{1,1},\dots, E_{1,\ell_1},E_{2,1},\dots,E_{2,\ell_2},\dots, E_{k,1},\dots,E_{k,\ell_k}\bigr\rangle
\end{equation}
such that for every $i=1,\dots,k$, there is a transitive $G$-action on the index set $\{1,\dots,\ell_i\}$ with the property that $g_*E_{i,\ell_i}\cong E_{i,g(\ell_i)}$.\ Let $H_i=\Stab_{G}(1)$ be the stabiliser of the first member of the $i$-th block.\ Assume that for every $i=1,\dots,k$, there exists an $H_i$-equivariant object $\cE_i\in D^b_{H_i}(X)$ with underlying non-equivariant object $E_{i,1}$.\ Then there is an induced full exceptional collection on $D^b_G(X)$, namely
\[
D^b_G(X)=\Bigl\langle \bigl(\Ind_{H_1}^G(\rho \otimes \cE_1)\bigr)_{\rho\in \irr(H_1)}, \bigl(\Ind_{H_2}^G(\rho \otimes \cE_2)\bigr)_{\rho\in \irr(H_2)}, \dots, \bigl(\Ind_{H_k}^G(\rho \otimes \cE_k)\bigr)_{\rho\in \irr(H_k)}\Bigr\rangle\,.
\]
\end{theorem}

For some of the occurring notions, in particular the induction functors $\Ind_{H_i}^G$, we have to refer to \autoref{subsect:equicat} below.\ For the purpose of this introduction it suffices to keep in mind that an exceptional collection on $D^b(X)$ fulfilling some compatibility with the $G$-action yields an exceptional collection on $D^b_G(X)$ in a canonical way.\ This has been used quite a lot to construct new exceptional collections.\ For example, it follows that an exceptional collection on a smooth projective surface $S$ canonically induces exceptional collections on the Hilbert schemes of points $S^{[n]}$ for every positive integer $n$; see \cite[Sect.\ 4]{KSos} and \autoref{subsect:symmquotstacks}.

There are at least two main reasons why exceptional collections are studied.\ Firstly, they induce the finest possible semi-orthogonal decompositions of derived categories.\ The second reason is that if they are \emph{strong} (i.e.\ they fulfill an additional Ext-vanishing condition, see \autoref{def:EC}(iv)), they induce an equivalence between $D^b(X)$ (or $D^b_G(X)$) with the bounded derived category $D^b(A)$ of some path algebra $A$ of a quiver with relations; see \cite{Bon}.\ More precisely, $A$ is the endomorphism algebra of the \emph{tilting object} given by the direct sum of all members of the exceptional collection.

It is easy to check that if the exceptional collection \eqref{eq:originalECintro} is strong, the same holds for the induced
exceptional collection on $D^b_G(X)$.\ Hence it is desirable to compute the endomorphism algebra of its direct sum.\ In particular one can ask: 
\begin{center}
\begin{small}
\enquote{\emph{Is there a universal formula expressing the endomorphism algebra of the induced equivariant collection in terms of the endomorphism algebra of the original exceptional collection?}}
\end{small}
\end{center}
\vspace{3pt}
Our main result is an affirmative answer to this question, namely:
\begin{theorem}\label{thm:intro2}
Under the assumptions of \autoref{thm:intro1}, let $\IE$ denote the direct sum of the members of the exceptional collection \eqref{eq:originalECintro}, and let $\IF$ denote the direct sum of the members of the induced exceptional collection on $D^b_G(X)$.\ Assume also that
\eqref{eq:originalECintro} is strong.\ Then
\[
\End(\IF) \cong \bigl(G\ltimes \End(\IE)\bigr)^b  \,.
\] 
\end{theorem}
Here $G\ltimes \End(\IE)$ stands for the \emph{skew group algebra} and
$\bigl(G\ltimes \End(\IE)\bigr)^b$ denotes its \emph{basic reduction}, i.e.\ the unique basic algebra Morita equivalent to $G\ltimes \End(\IE)$.\ See \autoref{subsect:basic} and \autoref{subsect:skew} for details on these notions.\ 

Let us outline the proof.\ As a path algebra with relations, $\End(\IF)$ is automatically basic.\ See \autoref{lem:ECbasic} for a slight generalisation of this fact to weakly exceptional sequences.\ Hence it suffices to prove that $\End(\IF)$ is Morita equivalent to the skew group algebra.\ For this purpose, the main observation is that there is actually a much easier construction than \autoref{thm:intro1} to turn a linearisable tilting object $\IE$ of $D^b(X)$ into a tilting object of $D^b_G(X)$ - just by applying the induction functor; see \autoref{cor:equitilt}. In particular, $\Ind(\IE)$ is a tilting object of $D^b_G(X)$.\ We compute the endomorphism algebra of this tilting object to be the skew group algebra $G\ltimes \End(\IE)$ in \autoref{cor:specialequitilt}.\ Now the key step is \autoref{lem:comparisonsummand}, proving that $\IF$ and $\Ind(\IE)$ have the same indecomposable summands in their Krull--Remak--Schmidt decompositions (although the summands show up in $\Ind(\IE)$ with higher multiplicities).\ From this it follows that under the equivalence $D^b_G(X)\cong D^b(\End(\IF))$ induced by the exceptional sequence, $\Ind(\IE)$ gets mapped to a progenerator of $\Mod(\End(\IF))\subset D^b(\End(\IF))$.\ This implies the desired Morita equivalence between $\End(\IF)$ and $\End\bigl(\Ind(\IE)\bigr)\cong G\ltimes \End(\IE)$.\par \vspace{0.3cm}

Actually, we work in greater generality than that of \cite{Ela} and prove a straightforward generalisation of \autoref{thm:intro1}, namely \autoref{thm:fullEC}.\ Instead of working with derived categories of coherent sheaves on varieties, we work with dg-enhanced triangulated categories; see \autoref{subsect:setupcat} for the details.\ Furthermore, we remove the assumption that $\fk$ is algebraically closed (but need to keep the assumption that $\Char\fk$ is coprime to $|G|$).\

The only difference if $\fk$ is not algebraically closed is that the induced equivariant collection is only \emph{weakly} exceptional; see \autoref{def:EC}(i) for this notion.\ Even then, the description of its endomorphism algebra as the basic reduction of the skew group algebra remains unchanged.\ See \autoref{thm:main}, which is the analogue of \autoref{thm:intro2} in the generalised set-up.

That the induced collection in our generalised set-up is still (weakly) exceptional is a straightforward computation, similar to computations of \cite{Ela}; see \autoref{lem:inducedEC}.\ The fullness of the collection follows from the comparison with the tilting object $\Ind(\IE)$ in \autoref{lem:comparisonsummand} that has to be carried out anyway to obtain the main result.

Parts of the results of this paper were already achieved in E.\ N.'s Master thesis under the supervision of A.\ K.

\section{Preliminaries}

\subsection{Categorical set-up}\label{subsect:setupcat}

Let $\fk$ be a field.\
Throughout, let $\cD$ denote a $\fk$-linear triangulated category which is \emph{cocomplete}.\ This ensures that arbitrary set-valued direct sums (not only finite ones) in $\cD$ exist.\ Furthermore, we assume that the category $\cD$ has a dg-enhancement $\widehat \cD$.\ This is a pretriangulated dg-category with $[\widehat \cD]=\cD$. For details on dg-enhanced categories, see for example \cite[Sect.\ 3]{KuzLunts}.

For $E,F\in \cD$, we write $\Hom^*(E,F) := \bigoplus_{i\in \IZ}\Hom_{\cD}(E,F[i])$.\ For $E, F\in \widehat\cD$, we write $\Hom^\bullet(E, F):=\Hom_{\widehat\cD}(E, F)$ which is a complex of vector spaces.\ Note that the objects in $\widehat \cD$ and $[\widehat\cD]=\cD$ are the same and that furthermore $\Hom^*(E,F)\cong \mathcal H^*\bigl(\Hom^\bullet(E,F)\bigr)$.  

We denote by $\cD^c$ the thick triangulated subcategory of \emph{compact} objects, i.e. those objects $F\in \cD$ such that $\Hom(F,\_)$ commutes with arbitrary direct sums.\ See e.g.\ \cite[Sect.\ 2.3]{KuzLunts} for details.\ From \autoref{subsect:basic} on, we make the assumption that $\cD^c$ is \emph{Hom-finite} (also called \emph{proper}), meaning that $\Hom^*(E,F)$ is a finite-dimensional $\fk$-vector space for all $E,F\in \mathcal D^c$. 

\begin{examples}\label{ex:D}
Let us list some classes of categories $\cD$ meeting the above assumptions.
\begin{enumerate}
\item If $\cG$ is a $\fk$-linear Grothendieck abelian category, its derived category $\cD=D(\cG)$ is a dg-enhanced cocomplete $\fk$-linear triangulated category.\ A dg-enhancement is given by the category of h-injective complexes over $\cG$; see \cite[Ex.\ 3.4]{CS--dgtour}.\ Two particularly important special cases of this are the following.
\begin{itemize}
\item If $X$ is a seperated scheme of finite type over $\fk$, then $\cG=\QCoh(X)$ is a $\fk$-linear Grothendieck category.\ By \cite[Thm.\ 3.1.1 \& 3.1.3]{BvdB--generators}, the subcategory of compact objects $D(\cG)^c$ agrees with the subcategory of perfect complexes in $D(\cG)$.\ In particular, if $X$ is proper, $D(\cG)^c$ is Hom-finite.\ If $X$ is smooth, then $D(\cG)^c\cong D^b(\Coh(X))$ agrees with the bounded derived category of coherent sheaves.\ In particular, our assumptions are fulfilled in the set-up of \cite{Ela}.
\item Let $A$ be a $\fk$-algebra.\ Our convention is that $A$-modules are always \emph{right} $A$-modules (in particular, $\End_A(A) \cong A$, not $A^{\mathsf{op}}$, as $\fk$-algebras).\ The $A$-modules form a $\fk$-linear Grothendieck category $\cG=\Mod(A)$.\ The subcategory of compact objects $D(\cG)^c$ agrees with the subcategory of perfect complexes in $D(\cG)$; see \cite[\href{https://stacks.math.columbia.edu/tag/07LQ}{07LQ}]{stacks-project}.\ If $A$ is finite-dimensional, $D(\cG)^c$ is Hom-finite.\ If $A$ is of finite global dimension, then $D(\cG)^c\cong D^b(\fdMod(A))$ agrees with the bounded derived category of finite-dimensional $A$-modules.\
\end{itemize}
\item If $\cA$ is a dg-category over $\fk$, its derived category $D(\cA)$ is a dg-enhanced cocomplete $\fk$-linear triangulated category; see e.g.\ \cite[Sect.\ 3.4]{KuzLunts} for details.\ In particular, this includes the derived categories of dg-algebras, which are the same as dg-categories with one object.\ If we specialise further to an ordinary algebra, i.e.\ a dg-algebra $\cA=A$ concentrated in degree $0$, we get back $D(A)=D(\Mod(A))$, where $\Mod(A)$ is the Grothendieck category of ordinary (non dg) modules over $A$.
\end{enumerate}
\end{examples}

From \autoref{subsect:equicat} on, where an action of a finite group $G$ on the category $\cD$ enters the stage, we add the assumption that $\Char \fk$ is zero or coprime to $|G|$.

\subsection{Generators, Tilting Objects and Exceptional Collections}

\begin{definition}
A \emph{compact generator} of $\cD$ is an object $G\in \cD^c$ such that $\Hom^*(G,F)= 0$ only holds for $F\cong 0$.\ This can also be written as $\langle G \rangle^{\perp} = 0$.  
\end{definition}

\begin{prop} \label{DGequivalence}
Let $G$ be a compact generator of $\cD$ and consider the dg-algebra defined as $\cB:=\Hom^{\bullet}(G,G)$.
Then the functor 
\[
\Phi=\Hom^\bullet(G, \_)\colon \cD\xrightarrow\cong D(\cB)  
\]
is an equivalence.
\end{prop}

More precisely, $\Hom^\bullet(G, \_)$ is a dg functor $\widehat \cD\to \dgMod(\cB)$ to begin with, inducing a functor $\cD=[\widehat \cD]\to [\dgMod(\cB)]$ on the level of the homotopy categories.\ Then $\Phi$ is the composition of this functor with the Verdier quotient $[\dgMod(\cB)]/[\Ac(\cB)]=D(\cB)$.

\begin{proof}
This is \cite[App.\ B]{LS-enhancements} with $P=I=G$ and $z=\id_G$.\ For more classical references, see e.g.\ \cite[Sect.\ 4.2]{Keller-derivingdg} or \cite[Prop.\ 1.16 \& 1.17]{Lunts-Orlov}.
\end{proof}

\begin{remark}
The equivalence $\Phi$ preserves arbitrary direct sums.\ Consequently, it restricts to an equivalence between the triangulated subcategories of compact objects
\[
 \Phi\colon \cD^c\xrightarrow\cong D(\cB)^c=\Perf(\cB)\,;
\]
see \cite[Lem.\ 2.10]{KuzLunts}.
\end{remark}

\begin{definition}
A \emph{tilting object} of $\cD$ is defined to be a compact generator $T\in \cD^c$ such that $\End^*(T)=\Hom^*(T,T)$ is concentrated in degree $0$.
\end{definition}

\begin{remark}\label{rem:ECtilt}
Let $T$ be a tilting object.\ Because the cohomology $\End^*(T)$ of the dg-algebra $\cB=\End^\bullet(T)$ is concentrated in degree $0$, we have a quasi-isomorphism between $\cB$ and the ordinary algebra $B=\End_{\cD}(T)$ without graded or dg structure.\ This quasi-isomorphism induces an equivalence $D(\cB)\cong D(B)$.\ Combining this with \autoref{DGequivalence} gives an equivalence  
\[
\Phi=\Hom^\bullet(T,\_)\colon \cD\xrightarrow \cong D(B)\,.
\]
Note that $\End^\bullet(T)=\Hom^\bullet(T,T)$ is quasi-isomorphic, hence isomorphic in $D(B)$, to the right $B$-module $B$.\ In other words, $\Phi(T)\cong B$.
\end{remark}

\begin{definition}\label{def:EC} Definitions (ii) - (v) below are standard.\ Definition (i) is not new, but also not among the most prominent.\ It occurs in \cite[Def.\ 1.16]{orlovglueing} and implicitly in \cite{ringelexceptional}.
\begin{enumerate}
 \item An object $E\in \cD^c$ is \emph{weakly exceptional} (or \emph{w-exceptional})\footnote{We will mostly use the abbreviation \emph{w-exceptional} in the following.\ The reason is not the desire for brevity, but rather that ``strong w-exceptional collection'' sounds better than ``strong weakly exceptional collection''.} if $\End^*(E)$ is concentrated in degree zero and $\End(E)$ is a finite-dimensional division algebra over $\fk$.
 \item An object $E\in \cD^c$ is \emph{exceptional} provided that
$\End^*(E)$ is concentrated in degree zero and $\End(E)=\mathbf k\cdot \id_E$.
\item A sequence $(E_1,\dots,E_\ell)$ of objects in $\mathcal D^c$ is called a \emph{(weakly) exceptional sequence} if all $E_i$ are (weakly) exceptional, and the semi-orthogonality condition $\Hom^*(E_i,E_j)=0$ for all $i>j$ holds.
\item A (weakly) exceptional sequence $(E_1,\dots,E_\ell)$ is \emph{strong} if
$\Hom^*(E_i,E_j)$ is concentrated in degree zero for all $i<j$ (note that for $i\ge j$, this holds anyway).
\item A (weakly) exceptional sequence $(E_1,\dots,E_\ell)$ is said to be \emph{completely orthogonal} if
$\Hom^*(E_i,E_j)=0$ for all $i<j$ (note that for $i> j$, this holds anyway).
\item A (weakly) exceptional sequence $(E_1,\dots,E_\ell)$ is \emph{full} if the direct sum of its members $\IE=E_1\oplus \dots \oplus E_\ell$ is a compact generator of $\cD$.
\end{enumerate}
\end{definition}

\begin{remark}\label{rem:PhiKRS}
By definition, a w-exceptional sequence $(E_1,\dots,E_\ell)$ is full if and only if the object
$\IE=E_1\oplus \dots \oplus E_\ell$ is a compact generator.\ It is both full and strong if and only if $\IE$ is a tilting object.\ In particular, for a full, strong w-exceptional sequence, we have an equivalence
\[
\Hom^\bullet(\mathbb E,\_)\colon \cD\xrightarrow \cong D(A),\quad \quad A:=\End(\IE)\,.
\]
As the $\End(E_i)$ are division algebras, the objects $E_i$ are indecomposable.\ Thus the same holds true for its images $\Phi(E_i)\cong \Hom(\mathbb{E},E_i)\in \Mod(A)\subset D(A)$ which are isomorphic to $A$-modules, i.e.\ to complexes concentrated in degree zero.\ Consequently,
\[
 A\cong \Phi(\IE)\cong \Phi(E_1)\oplus \dots \oplus \Phi(E_\ell)
\]
is the Krull--Remak--Schmidt decomposition of the right $A$-module $A$.  
\end{remark}

\begin{remark}\label{rem:fingd}
In the literature, the additional assumption that $A=\End(T)$ is of finite global dimension is often made for a tilting object $T$.\ This ensures that the equivalence $\Phi$ maps $\cD^c$ to $D^b(\fdMod(A))$, cf.\ \autoref{ex:D}.\ If $T=\IE$ is the direct sum of a strong full exceptional collection, this is automatic:\ In that case, $A$ is an admissible quotient of a path algebra.
\end{remark}

\subsection{Basic algebras}\label{subsect:basic}

A finite-dimensional algebra $A$ is \emph{basic} if the indecomposable summands of the Krull--Remak--Schmidt decomposition $A\cong P_1\oplus\dots\oplus P_\ell$ of $A$ as a right module over itself are pairwise non-isomorphic.\

\begin{lemma}\label{lem:ECbasic}
Let $E_1,\dots, E_\ell$ be a strong w-exceptional collection and let $\IE:=E_1\oplus \dots\oplus E_\ell$.\ Then the endomorphism algebra $A=\End(\IE)$ is a basic $\fk$-algebra.
\end{lemma}

\begin{proof}
By \autoref{rem:PhiKRS}, the Krull--Remak--Schmidt decomposition of $A$ is
\[
 A \cong \Phi(E_1)\oplus \dots \oplus \Phi(E_\ell)\,.
\]
By the semi-orthogonality in a $w$-exceptional collection, the $E_i$ are pairwise non-isomorphic.\ As $\Phi$ is an equivalence, the same remains true for the summands $\Phi(E_i)$ of the Krull--Remak--Schmidt decomposition.
\end{proof}

\begin{propdef}\label{propdef:basic}
For every finite-dimensional $\fk$-algebra $A$, there exists a unique basic algebra $A^b$ up to isomorphism which is Morita equivalent to $A$, meaning that $\Mod(A)\cong \Mod(A^b)$.\ We call $A^b$ the \emph{basic reduction} of $A$.  
\end{propdef}

\begin{proof}
The result is due to \cite[Thm.\ 7.5]{Morita} and proofs can be found in most textbooks on the representation theory of algebras.\ Note that $A^b$ can computed very explicitly.\ For example, one can consider the Krull--Remak--Schmidt decomposition
\[
 A\cong P_1^{\oplus n_1}\oplus \dots \oplus P_\ell^{\oplus n_\ell}\qquad \text{with $P_i\not\cong P_j$ for $i\neq j$} 
\]
of $A$ as a right module over itself.\ Then $A^b\cong \End_A(P_1\oplus\dots\oplus P_\ell)$. 
\end{proof}

\subsection{Group actions on categories, equivariant categories}\label{subsect:equicat}

From now on, let $G$ denote a finite group, and assume that $\Char\fk$ is zero or coprime to $|G|$.\ We discuss the generalities on categorical group actions and the associated equivariant categories only very briefly.\ For details we refer to \cite{BO--equivariant, Elagin--onequi, Shinder--equivariant}.\ We follow the convention of \cite{BO--equivariant, Shinder--equivariant} that the group acts on the category from the left (covariantly). However, one can easily switch to a right-action as in \cite{Elagin--onequi} by replacing every $g\in G$ by its inverse $g^{-1}$.

An \emph{action} of $G$ on a category $\cC$ is given by a family $\bigl(g_*\colon \cC\xrightarrow\sim \cC\bigr)_{g\in G}$ of autoequivalences together with isomorphisms of functors 
\begin{align*}
 \bigl(\eps_{g,h}\colon g_*\circ h_*\xrightarrow \sim (gh)_*\bigr)_{g,h\in G}
\end{align*}
satisfying a compatibility condition.

One main example is the following.\ Let $X$ be a scheme and $G\le \Aut(X)$.\ Then $G$ acts on $\QCoh(X)$ by pushforwards of sheaves, inducing also a $G$-action on $D(X)$.\ This is the set-up in which \cite{Ela} is formulated, with the only difference that in \emph{loc.\ cit.}\ the group acts by pull-backs, giving a right-action of $G$ on $D(X)$.

Given a $G$-action on $\cC$, there is the associated \emph{equivariant} category $\cC^G$.\ Its objects are pairs $(E,\lambda)$, where $E\in \cC$ and $\lambda$ is a \emph{$G$-linearisation}, i.e.\ a family $\bigl(\lambda_g\colon E\xrightarrow \sim g_*E \bigr)_{g\in G}$ of isomorphisms satisfying the condition $\eps_{g,h}\circ g_*\lambda_h\circ \lambda_g= \lambda_{gh}$.\par
\noindent Given two such pairs $(E,\lambda)$ and $(F,\mu)$, there is a left-action of $G$ on $\Hom_\cC(E,F)$ by conjugation with the linearisations, more precisely
\[
 {}^g\phi:=\mu_g^{-1}\circ g_*\phi\circ \lambda_g\quad\text{for $\phi\in \Hom_\cC(E,F)$ and $g\in G$.}
\]
We turn this into a right-action by setting $\varphi^g := {}^{g^{-1}}\varphi$.\ The reason is that in \autoref{subsect:skew} below, right-actions of groups on algebras will be considered.\ This should include endomorphism algebras of equivariant objects.\ In any case, the morphisms in $\cC^G$ are given by invariants under the action of $G$:
\begin{equation}\label{eq:HomG}
 \Hom_{\cC^G}\bigl((E,\lambda),(F,\mu) \bigr):=\Hom_\cC(E,F)^G.
\end{equation}

Let now $\cC=\cD$ be a dg-enhanced cocomplete $\fk$-linear triangulated category (cf.\ the assumptions in \autoref{subsect:setupcat}).\ Let us assume furthermore that the dg-enhancement $\widetilde \cD$ can be chosen in a way so that the $G$-action on $\cD$ is induced by an action $\{\hat g\colon \widehat \cD\xrightarrow\sim \widehat \cD\}_{g \in G}$ by dg-autoequivalences $\hat g$, meaning that $[\hat g]=g$ for all $g\in G$.

This is fulfilled if $\cD=D(\cA)$ is the derived category of some Grothendieck category $\cA$ and the $G$-action on $\cD$ is induced by a $G$-action on $\cA$.\ This in turn happens for the example of finite subgroups of automorphisms of schemes discussed above, and also in the example of finite automorphism groups of algebras discussed in \autoref{subsect:skew} below.

Under all these assumptions, the equivariant category $\cD^G$ turns out to be again a dg-enhanced \cite[Cor.\ 8.10]{Elagin--onequi} cocomplete (there is an obvious way to form direct sums in the equivariant category) $\fk$-linear triangulated \cite[Cor.\ 6.10]{Elagin--onequi} category.

Given a subgroup $H\le G$, there is the \emph{restriction} functor 
\[
 \Res_G^H\colon \cD^G\to \cD^H,\quad\quad (E,\lambda)\mapsto (E,\lambda_{\mid H}),
\]
having the \emph{induction} functor $\Ind_H^G\colon \cD^H\to \cD^G$ as a left and right adjoint.\  We have the description 
\begin{equation}\label{eq:Ind}
\Ind_H^G(E)=\bigoplus_{[g]\in G/H} g_*E
\end{equation}
where the sum runs through a set of representatives of the cosets, and
the $G$-linearisation is given by a combination of the $H$-linearisation of $E$ and the isomorphisms $\eps_{g,h}$ between the appropriate direct summands. See \cite[Sect.\ 3.2]{BO--equivariant} for details.\ When the groups $H$ and $G$ are clear from the context\footnote{Often $H=1$ is the trivial subgroup of $G$.}, we often simply write $\Res$ and $\Ind$ instead of $\Res_G^H$ and $\Ind_H^G$.

\begin{lemma}\label{lem:gInd}
Let $E\in \cD$.\ Then $\Ind_1^G(E)\cong \Ind_1^G(g_*E)$ naturally for every $g\in G$.  
\end{lemma}

\begin{proof}
There is a natural isomorphism $\Res\cong (g^{-1})_*\circ \Res$, given on objects $\cF=(F,\lambda) \in \cD^G$ by the linearisation 
\[
 \Res(\cF)=F\xrightarrow{\lambda_{g^{-1}}} (g^{-1})_*F=(g^{-1})_*\Res(\cF)\,.
\]
Taking adjoints, we get an isomorphism of functors $\Ind\cong \Ind\circ g_{\ast}$. 
\end{proof}

\begin{remark}\label{rem:tensorG}
Since $\cD$ is $\fk$-linear, we have a well-defined tensor product $V\otimes_{\fk} E$ for every $E\in \cD$ and every finite-dimensional vector space $V$, namely $V\otimes_{\fk} E\cong E^{\oplus \dim V}$. 

More canonically, the tensor product is the object satisfying the two universal properties $\Hom_{\cD}(V\otimes_{\fk} E,F)\cong \Hom_{\fk}\bigr(V,\Hom_{\cD}(E,F)\bigl)$ and $\Hom_{\cD}(F,V\otimes_{\fk} E)\cong V\otimes_\fk \Hom_{\cD}(E,F)$.\
Thus, for objects $E,F\in \cD$ and finite-dimensional vector spaces $V,W\in \mathsf{vec}_\fk$, we have
\begin{equation}\label{eq:tensoruniv}
\Hom_{\cD}(V\otimes_{\fk} E, W\otimes_{\fk} F)\cong \Hom_{\fk}(V,W)\otimes_\fk \Hom_{\cD}(E,F)\,.
\end{equation}
Let now $V=(V,\rho)\in \rep(G)$ be a finite-dimensional right-G-representation, given by a  group homomorphism $\rho\colon G\to \Aut_{\fk}(V)^{\mathsf{op}}$.\ Then for each $\cE=(E,\lambda)\in \cD^G$, there is a well-defined equivariant object $\rho \otimes_\fk \cE := V\otimes_\fk \cE:=(V,\rho)\otimes_\fk \cE$.\ Its underlying non-equivariant object is $V\otimes_\fk E$ and the linearisation is given by 
\[
\rho_g\otimes \lambda_g \in \End_{\fk}(V)\otimes_{\fk} \Hom(E,g_*E)\cong \Hom(V \otimes_{\fk} E,V \otimes_{\fk} (g_*E))\cong \Hom(V \otimes_{\fk} E,g_*(V \otimes_{\fk} E))\,.     
\]
This means that we have an action of the monoidal category
 $\rep(G)$ on the category $\cD^G$.
\end{remark}

\subsection{Skew groups algebras}\label{subsect:skew}

Let now $A$ be a $\fk$-algebra, and let $G$ be a finite group acting on $A$ from the right via $\fk$-algebra automorphisms.\ For $g\in G$ and $a\in A$, we write $a^g:=a\cdot g$ so that $(ab)^g=a^gb^g$ and $(a^g)^h=a^{gh}$.\ The group $G$ then acts on the category $\Mod(A)$ of right $A$-modules as follows.\ For $M\in \Mod(A)$ and $g\in G$, the module $g_*M$ has the same underlying abelian group as $M$ but the new scalar multiplication $\star$ differs by $g$ from the original one:\ $m\star a=m\cdot a^g$. Note that this is a \emph{strict} categorical left-action, meaning that the isomorphisms $\eps_{g,h}\colon g_*\circ h_*\xrightarrow \sim (gh)_*$ are the identities.

Let us now define the \emph{skew group} algebra $G\ltimes A$; cf.\ \cite{RR--skew}.\ As a $\fk$-vector space, it is defined by having a basis consisting of pairs $(g,a)$ with $g\in G$, and $a\in A$.\ In other words, $G\ltimes A$ coincides with $\fk\langle G\rangle \otimes_{\fk} A$ as a vector space, where $\fk\langle G\rangle$ is the regular representation (or rather its underlying vector space).\ The multiplication is given by linear extension of
\[
 (g,a)\cdot (h,b)=(gh, a^h\cdot b)\,.
\]

\begin{lemma}\label{lem:skewequivalence}
There is a natural equivalence of categories $\Mod(G\ltimes A)\cong \Mod(A)^G$.
\end{lemma}

\begin{proof}
This is well-known; see e.g.\ \cite[Ex.\ 2.6]{Chen--equi} and the unpublished notes \cite{Poon}.\ Since the conventions in the literature differ from our ones in whether to consider left or right modules or actions, let us quickly recall the construction of the mutually inverse equivalences.\ This will be needed in \autoref{rem:skewkG} below in any case.\par 

Given an equivariant $A$-Module $(M,\lambda)$, the associated $(G\ltimes A)$-module has the same underlying abelian group, and is equipped with the scalar multiplication $m\cdot (g,a):=\lambda_{g}(m)a$.

Conversely, given a $(G\ltimes A)$-module, we can restrict it to an $A$-module via the embedding of algebras 
\[
 A \hookrightarrow G\ltimes A,\quad\quad a\mapsto (e,a)
\]
and equip this restricted $A$-module with the linearisation given by $\lambda_g(m):=m\cdot (g,1)$.
\end{proof}

\begin{remark}\label{rem:skewkG}
The right $A$-module $A$ has a natural $G$-linearisation given by its right-action $\lambda_g=g$.\ Under the equivalence of \autoref{lem:skewequivalence}, the right $(G\ltimes A)$-module $G\ltimes A$ corresponds to the equivariant $A$-module $\fk\langle G\rangle \otimes_\fk A$.\ See \autoref{rem:tensorG} for the definition of the tensor product with the regular (right-)representation $\fk\langle G\rangle$ of the finite group $G$.
\end{remark}

This allows us to describe the skew group algebra as the invariants of a tensor product of algebras with $G$-actions.

\begin{lemma}\label{lem:skewashominva}
We have an isomorphism of $\fk$-algebras
\[
G\ltimes A\cong \Bigl( \Hom_\fk\bigl(\fk\langle G\rangle, \fk\langle G\rangle\bigr) \otimes_\fk A \Bigr)^G\,.
\]
\end{lemma}

\begin{proof}
Indeed,
\begin{align*}
G\ltimes A &\cong \Hom_{G\ltimes A}(G\ltimes A, G\ltimes A)\\
& \cong \Hom_A\bigl(\fk\langle G\rangle \otimes_\fk A, \fk\langle G\rangle \otimes_\fk A\bigr)^G\tag{by \autoref{rem:skewkG} and \eqref{eq:HomG}}\\
& \cong \Bigl(  \Hom_\fk\bigl(\fk\langle G\rangle, \fk\langle G\rangle\bigr) \otimes_\fk \Hom_A(A,A) \Bigr)^G\tag{by \eqref{eq:tensoruniv}}\\
& \cong \Bigl(  \Hom_\fk\bigl(\fk\langle G\rangle, \fk\langle G\rangle\bigr) \otimes_\fk A \Bigr)^G\,.\qedhere
\end{align*}
\end{proof}

\section{Proof of the Main Theorem}

We continue with assuming that $\cD$ is a dg-enhanced cocomplete $\fk$-linear triangulated category on which a finite group $G$ with $\Char \fk$ and $|G|$ coprime acts, and that the action lifts to a dg-enhancement $\widehat\cD$ of $\cD$. 

\subsection{Induced generators and tilting objects in the equivariant category}

\begin{lemma}\label{lem:compact}
 If $G$ is a compact generator of $\cD$, then $\Ind(G)$ is a compact generator of $\cD^G$.
\end{lemma}

\begin{proof}
 Since its right adjoint $\Res$ commutes with arbitrary direct sums, compact objects are preserved by the functor $\Ind\colon \cD\to \cD^G$; see  \cite[Lem.\ 2.10]{KuzLunts}.\ In particular, $\Ind(G)$ is compact.
 
Let $\cF\in \cD^G$ with $\Hom^*_{\cD}(\Ind G,\cF)=0$. Then, by adjunction, $\Hom^*_\cD(G,\Res\cF)=0$, which implies $\Res\cF\cong 0$ since $G$ is a generator. As $\Res$ is just given by forgetting the linearisation, this already implies that $\cF\cong 0$.  
\end{proof}

Next, we compute the endomorphism algebras of objects in the image of the functor $\Ind\circ \Res\colon \cD^G\to \cD^G$.\ To this end, let $\cT\in \cD^G$ and $T=\Res(\cT)\in \cD$.\ Notice that on the graded algebra $\End^*_{\cD}(T)$, there is a $G$-action given by conjugation with the $G$-linearisation of $\cT$. Hence, we can form the skew group algebra $G\ltimes \End^*_{\cD}(T)$.   

\begin{lemma}\label{cor:IndEnd}
Let $T\in \cD$ be \emph{$G$-linearisable}, which means that $T=\Res(\cT)$ for some object $\cT\in \cD^G$ in the equivariant category.\ Then 
\[
\End^*_{\cD^G}(\Ind(T))\cong G\ltimes \End^*_{\cD}(T)\,.\]
\end{lemma}

\begin{proof}
We have $\Ind(T)\cong \Ind(\Res(\cT))\cong \fk\langle G\rangle \otimes_{\fk} \cT$; see \cite[Prop.\ 4.1]{Elagin--onequi}.\ Hence
\begin{align*}
 \End^*_{\cD^G}(\Ind(T))&\cong \Hom^*_{\cD^G}\bigl(\fk\langle G\rangle \otimes_{\fk} \cT,\fk\langle G\rangle \otimes_{\fk} \cT \bigr)\\
&\cong \Bigl(  \Hom_\fk\bigl(\fk\langle G\rangle, \fk\langle G\rangle\bigr) \otimes_\fk  \End^*_\cD(T)\Bigr)^G\\ 
& \cong G\ltimes \End_{\cD}^*(T), 
\end{align*}
where the last isomorphism follows from \autoref{lem:skewashominva} with $A=\End^*(T)$.  
\end{proof}

\begin{cor}\label{cor:equitilt}
If $T=\Res\cT$ is a tilting object of $\cD$, then $\Ind(T)$ is a tilting object of $\cD^G$ fulfilling 
\[
\End_{\cD^G}(\Ind(T))\cong G\ltimes \End_{\cD}(T)\,.\]
\end{cor}

\subsection{The set-up and the induced equivariant exceptional sequence}\label{subsect:setup}

From now on, we assume that we have a full exceptional collection 
\begin{align}\label{eq:originalEC}
 (E_{1,1},\dots, E_{1,\ell_1},E_{2,1},\dots,E_{2,\ell_2},\dots, E_{k,1},\dots,E_{k,\ell_k})
\end{align}
on $\cD$ such that $G$ acts transitively on every block $(E_{i,1},\dots, E_{i,\ell_i})$ (compare this to (\ref{eq:originalECintro})).\ By this we mean that there exists a transitive $G$-action on the index set $\{1,\dots,\ell_i\}$ such that $g_*E_{i,\ell_i}\cong E_{i,g(\ell_i)}$.\ Let $H_i=\Stab_{G}(1)$ be the stabiliser of the first member $E_i:=E_{i,1}$ of the $i$-th block.\ Notice that the assumptions imply that each block $(E_{i,1},\dots, E_{i,\ell_i})$ is not only semi-orthogonal, but completely orthogonal. In particular, 
\begin{equation}\label{eq:fullortho}
\Hom^*_{\cD}(E_i,g_*E_i)=0\qquad\text{for all $i=1,\dots, k$ and $g\notin H_i$.} 
\end{equation}

We make the further assumption that $E_{i,1}$ is not only $H_i$-invariant, but also $H_i$-linearisable. This means that there exists an $\cE_i=(E_i,\lambda)\in \cD^{H_i}$ with
\[\Res_{H_i}^1(\cE_i)=E_i=E_{i,1}\,.\]
This implies that the direct sum of the members of every block is given by 
\begin{equation}\label{eq:blocklin}
 E_{i,1}\oplus\dots\oplus E_{i,\ell_1}\cong \Res_G^1\Ind_{H_i}^G \cE_i\,.
\end{equation}

Now, for $i=1,\dots, k$ and $\rho\in \irr(H_i)$ denoting an irreducible representation of the finite group $H_i \leq G$, we consider the $G$-equivariant object
\[
 F_{i,\rho}:=\Ind_{H_i}^G (\rho \otimes_{\fk}\cE_i)\,.
\]
\begin{lemma}\label{lem:inducedEC}
In $\cD^G$, the $F_{i,\rho}$ form a w-exceptional collection
\begin{equation}\label{eq:inducedEC}
\Bigl((F_{1,\rho})_{\rho\in \irr(H_1)}, (F_{2,\rho})_{\rho\in \irr(H_2)}, \dots, (F_{k,\rho})_{\rho\in \irr(H_k)}\Bigr)\,. 
\end{equation}
The members of each block $(F_{i,\rho})_{\rho\in \irr(H_i)}$ are completely orthogonal, hence their order does not matter.
If $\fk=\bar \fk$ is algebraically closed, then this is an exceptional collection.
\end{lemma}

\begin{proof}
We first prove that every single $F_{i,\rho}$ is weakly exceptional. We have

\begin{align*}
 \Hom^*_{\cD^G}(F_{i,\rho}, F_{i,\rho}) &= \Hom^*_{\cD^G}\bigl(\Ind_{H_i}^G (\rho \otimes_{\fk}\cE_i),\Ind_{H_i}^G (\rho \otimes_{\fk}\cE_i)\bigr)
\\
&\cong \Hom^*_{\cD^{H_i}}\bigl(\rho \otimes_{\fk}\cE_i, \Res_{G}^{H_i}\Ind_{H_i}^{G} (\rho \otimes_{\fk}\cE_i)\bigr) \tag{by adjunction}
 \\
& \cong \Hom^*_{\cD}\bigl(\rho\otimes_{\fk} E_i, \bigoplus{}_{[g]\in G/H_i} ~ g_*(\rho\otimes_{\fk} E_i)\bigr)^{H_i} \tag{by \eqref{eq:Ind}}
\\
 &\cong \Hom^*_{\cD}(\rho\otimes_{\fk} E_i,\rho\otimes_{\fk} E_i)^{H_i} \tag{by \eqref{eq:fullortho}}
 \\ &\cong
\Bigl(\Hom_{\fk}(\rho,\rho)\otimes_{\fk} \Hom^*_{\cD}(E_i,E_i)  \Bigr)^{H_i} \tag{by \eqref{eq:tensoruniv}}\\
&\cong
\Hom_\fk(\rho,\rho)^{H_i} \otimes_{\fk} \fk \cong \Hom_\fk(\rho,\rho)^{H_i}.
\end{align*}
The penultimate isomorphism is due to the fact that the $G$-action on the vector space
$\Hom^*_{\cD}(E_i,E_i)=\fk\cdot \id_{E_i}$ by conjugation with the $G$-linearisation is trivial.\ As $\rho$ is irreducible, Schur's lemma states that $\Hom_{\fk}(\rho,\rho)^{H_i}\cong \Hom_{H_i}(\rho,\rho)$ is a finite-dimensional  division $\fk$-algebra.\ So we proved that $F_{i,\rho}$ is weakly exceptional.\ In the case that $\fk$ is algebraically closed, $\Hom_{H_i}(\rho,\rho)=\fk\cdot \id$, so $F_{i,\rho}$ is exceptional. 
\\
Next, we prove complete orthogonality within each block $(F_{i,\rho})_{\rho\in\irr(H_i)}$.\ So let $\rho\not\cong \rho'$ be two non-isomorphic irreducible $H_i$-representations.\ Via the same computation as above, it follows that
$\Hom^*_{\cD^G}(F_{i,\rho}, F_{i,\rho'})\cong \Hom_{H_i}(\rho,\rho')$.\ This vector space vanishes, again by Schur's lemma.
\\
For the semi-orthogonality between the blocks, let $i>j$, $\rho\in\irr(H_i)$, and $\rho'\in \irr(H_j)$.\ 
Then, by  \eqref{eq:HomG} and \eqref{eq:blocklin}, $\Hom^*_{\cD^G}(F_{i,\rho}, F_{j,\rho'})$ is a vector subspace of
\[
\bigoplus_{\alpha=1}^{\ell_i}\bigoplus_{\beta=1}^{\ell_j}\bigl( \Hom_\fk(\rho,\rho')\otimes \Hom^*_{\cD}(E_{i,\alpha},E_{j,\beta})  \bigr). 
\]
By the semi-orthogonality of the original exceptional sequence \eqref{eq:originalEC}, all summands vanish.
\end{proof}

\begin{remark}
\autoref{lem:inducedEC} shows that exceptional collections in $\cD$ induce $w$-exceptional collections in $\cD^G$.\ This statement becomes wrong if the original sequence (\ref{eq:originalEC}) is only assumed to be $w$-exceptional.\ For example, consider the quaternions $\mathbb{H} \cong_{\mathbb{R}} \langle 1,i,j,k\rangle$ as a division algebra over $\fk = \mathbb{R}$.\ Then the symmetric group $\sym_3$ on three letters acts on $\mathbb{H}$ by permuting $i,j$ and $k$.\ Tensoring this linear group action with the sign representation $\mathfrak{a}_3$ yields a group action on $\mathbb{H}$ that is easily checked to be compatible with the algebra structure of $\mathbb{H}$.\ \par

\par Then the right $\mathbb{H}$-module $E_1=\mathbb{H}$ is a weakly exceptional object in $\Mod(\mathbb{H})\subset D(\mathbb H)$. It alone forms a full strong exceptional sequence in $D(\mathbb H)$, and $E_1=\mathbb{H}$ comes along with the obvious linearisation (\autoref{rem:skewkG}).\ 
However, $F_{1,\rho}=\rho\otimes_\IR \mathbb H$ where $\rho$ is the standard representation of $\sym_3$ is not weakly exceptional.

Indeed, as representations of $\sym_3$ over $\mathbb{R}$, we have $\mathbb{H} \cong \mathfrak{a}_3 \otimes (\mathbb{R}^{\oplus 2} \oplus \rho)$.\ An easy calculation shows that as $\sym_3$-representations,
\begin{align*}
\End^{\ast}_{\mathbb{H}}(\rho \otimes_{\mathbb{R}} \mathbb{H}) &\cong \rho^{\otimes 2}  \otimes_{\mathbb{R}} \mathbb{H} \cong  (\mathbb{R} \oplus \mathfrak{a}_3 \oplus \rho) \otimes_{\mathbb{R}} \mathfrak{a}_3 \otimes_{\mathbb{R}}   (\mathbb{R}^{\oplus 2} \oplus \rho)
\end{align*}
contains the trivial representation $\mathbb{R}$ three times as a summand.\ Hence $\End^{\ast}_{\mathbb{H}}(\rho \otimes_{\mathbb{R}} \mathbb{H})^{\sym_3}$ is $3$-dimensional, therefore not a division algebra over $\mathbb{R}$.

\end{remark}

\subsection{Comparison of the direct summands}

We consider the direct sum of the members of the full exceptional sequence \eqref{eq:originalEC} 
\[
 \IE:= E_{1,1}\oplus\dots\oplus E_{1,\ell_1}\oplus E_{2,1}\oplus\dots\oplus E_{2,\ell_2}\oplus\dots\oplus E_{k,1}\oplus\dots\oplus E_{k,\ell_k}\,.
\]
It is a compact generator of $\cD$, and a tilting object if \eqref{eq:originalEC} is strong.\
Furthermore, $\IE$ is $G$-linearisable by \eqref{eq:blocklin}.\ Hence $T=\IE$ fulfills the assumptions of \autoref{cor:equitilt}.\ Together with \autoref{lem:compact}, this yields the following. 

\begin{cor}\label{cor:specialequitilt}
$\Ind(\IE)$ is a compact generator of $\cD^G$.\ If \eqref{eq:originalEC} is strong, then $\Ind(\IE)$ is a tilting object of $\cD^G$ with 
\[
\End_{\cD^G}(\Ind(\IE))\cong G\ltimes \End_{\cD}(\IE)\,.\]
\end{cor}

\begin{remark}
If \eqref{eq:originalEC} is strong, then the algebra $\End_\cD(\IE)$ is of finite global dimension; cf.\ \autoref{rem:fingd}.\ Consequently, the same holds true for the skew group algebra $\End_{\cD}(\IE)\rtimes G$; see \cite[Thm.\ 1.1 \& 1.3(c)(i)]{RR--skew}.\ This implies that $\Ind(\IE)$ is also a tilting object in the slightly stronger sense discussed in \autoref{rem:fingd}.   
\end{remark}

We also consider the direct sum of the members of the exceptional sequence \eqref{eq:inducedEC}
\[
\IF:=\bigoplus_{i=1}^k\bigoplus_{\rho\in \irr(H_i)} F_{i,\rho} 
\]
and want to compare this direct sum $\mathbb{F}$ to $\Ind(\IE)$.

\begin{lemma}\label{lem:comparisonsummand}
There is a family of positive integers $(n_{i,\rho})^{i=1,\dots,k}_{\rho\in \irr(H_i)}$ such that 
\[
\Ind(\IE)\cong \bigoplus_{i=1}^k\bigoplus_{\rho\in \irr(H_i)} F_{i,\rho}^{\oplus n_{i,\rho}}\,. 
\]
In particular, $\IF$ is a direct summand of $\Ind(\IE)$, and conversely $\Ind(\IE)$ is a direct summand of $\IF^{\oplus n}$ for some large enough integer $n$.
\end{lemma}

\begin{proof} Observe first that according to \autoref{lem:gInd}, there is an isomorphism
\begin{equation}\label{eq:step1}
 \Ind(\mathbb E)\cong \bigoplus_{i,j} \Ind(E_{i,j})\cong \bigoplus_i \Ind(E_{i,1})^{\oplus [G:H_i]}\,. 
\end{equation}
Next, note that all irreducible $H_i$-representations show up as direct summands in the regular representation:
\[
 \fk\langle H_i\rangle\cong \bigoplus_{\rho\in \irr(H_i)}\rho^{\oplus m_{i,\rho}}\qquad\text{for some $m_{i,\rho}>0$}.
\]
Concretely, $m_{i,\rho}$ is equal to $\frac{\dim_{\fk}\rho}{\dim_{\fk}\End_{H_i}(\rho)}$. 
Now, using again \cite[Prop.\ 4.1]{Elagin--onequi} gives
\begin{align*}
\Ind(E_{i,1})=\Ind_1^G \Res^1_{H_i} \cE_i &\cong \Ind_{H_i}^G\Ind_1^{H_i} \Res^1_{H_i} \cE_i \\
&\cong \Ind_{H_i}^G \bigl(\fk\langle H_i \rangle\otimes_{\fk} \cE_i\bigr) \\
&\cong\bigoplus_{\rho\in \irr(H_i)} \Ind_{H_i}^G \bigl((\rho\otimes_{\fk} \cE_i)^{\oplus m_{i\rho}}\bigr) \\
&\cong \bigoplus_{\rho\in \irr(H_i)} F_{i,\rho}^{\oplus m_{i\rho}} \,.
\end{align*}
Plugging this into \eqref{eq:step1} gives the assertion with $n_{i,\rho}=[G:H_i]m_{i,\rho}=\frac{[G:H_i]\dim_{\fk}\rho}{\dim_{\fk}\End_{H_i}(\rho)}$.
\end{proof}

\subsection{Proof of the main results}

\begin{theorem}\label{thm:fullEC}
The induced exceptional collection \eqref{eq:inducedEC} in $\cD^G$ of \autoref{lem:inducedEC}
is full.\ If the original exceptional collection \eqref{eq:originalEC} is strong, the same holds for \eqref{eq:inducedEC}.
\end{theorem}

\begin{proof}
By \autoref{lem:comparisonsummand}, $\IF$ is a direct summand of $\Ind(\IE)$.  
As $\Ind(\IE)$ is compact by \autoref{cor:specialequitilt}, the same then holds for $\IF$.

Again by \autoref{cor:specialequitilt} and \autoref{lem:comparisonsummand}, 
$\Ind(\IE)$ is a compact generator of $\cD^G$ and is a direct summand of some power of $\IF$.\ Hence also $\IF$ is a compact generator of $\cD^G$, showing fullness of \eqref{eq:inducedEC} by \autoref{rem:PhiKRS}.

Similarly, \autoref{cor:specialequitilt} and \autoref{lem:comparisonsummand}
show that $\IF$ is a tilting object of $\cD^G$ if the original exceptional sequence \eqref{eq:originalEC} is strong. Again by \autoref{rem:PhiKRS}, this means that also \eqref{eq:inducedEC} is strong in this case.
\end{proof}

\begin{theorem}\label{thm:main}
If \eqref{eq:originalEC} is strong, the endomorphism algebra of the direct sum of the induced strong exceptional sequence \eqref{eq:inducedEC} is the basic reduction
\[
 \End(\mathbb F)\cong \bigr(G\ltimes \End(\mathbb E)\bigl)^b\,.
 \]
\end{theorem}

\begin{proof}
By \autoref{rem:PhiKRS} and \autoref{thm:fullEC}, an equivalence $\Phi:=\Hom^\bullet(\IF,\_)\colon \cD\xrightarrow\cong D(\End(\IF))$ exists with the property that
\[
 \End(\IF)\cong \bigoplus_{i,\rho} \Phi(F_{i,\rho})
\]
is the Krull--Remak--Schmidt decomposition of $\End(\IF)$ as a right module over itself.\ According to \autoref{lem:comparisonsummand}, we have 
\[
 \Phi(\Ind\mathbb E)\cong \bigoplus_{i,\rho} \Phi(F_{i,\rho})^{\oplus n_{i,\rho}}\quad\text{for some $n_{i,\rho}\ge 1$}\,.
\]
This means that $\Phi(\Ind\mathbb E)$ is a progenerator in the category of right $\End(\IF)$-modules.\ Hence, $\End(\IF)$ and $\Hom_{\End(\IF)}\bigl(\Phi(\Ind\mathbb E), \Phi(\Ind\mathbb E)\bigr)$ are Morita equivalent. By fully faithfulness of $\Phi$ together with \autoref{cor:specialequitilt}, the latter algebra is isomorphic to the skew group algebra:
\[
\Hom_{\End(\IF)}\bigl(\Phi(\Ind\mathbb E), \Phi(\Ind\mathbb E)\bigr)\cong \End_{\cD^G}(\Ind \IE)\cong  G\ltimes \End(\IE)\,.
\]
So $\End(\IF)$ is Morita equivalent to $G\ltimes \End(\IE)$ and by \autoref{lem:ECbasic} it is basic.\ Hence the result follows from \autoref{propdef:basic}.
\end{proof}

\section{Further Remarks and Examples}

\subsection{Morita equivalence vs.\ derived equivalence}

As explained in the introduction, the main point of our proof of \autoref{thm:main} is to construct a Morita equivalence between $\End(\IF)$ and the skew group algebra $\End(\IE)\rtimes G$.

At least a derived equivalence $D(\End(\IF))\cong
D(\End(\IE)\rtimes G)$ already follows from general theory, without any need for the concrete computations in the proofs of \autoref{lem:skewashominva}, \autoref{cor:IndEnd}, and \autoref{lem:comparisonsummand}:

By \autoref{rem:ECtilt}, there is an equivalence
$\Psi:=\Hom^\bullet(\IE,\_)\colon \cD\xrightarrow\sim D(\End(\IE))$.
One checks that this induces an equivalence of the associated $G$-equivariant categories $\cD^G\cong D(\End(\IE))^G$.\ Combining this with \autoref{lem:skewequivalence} gives $\cD^G\cong D(G \ltimes \End(\IE))$.\ Under this equivalence, the tilting object $\IF\in \cD^G$ corresponds to some tilting object $\IF'\in D(G \ltimes \End(\IE))$ with the same endomorphism algebra. By \autoref{rem:ECtilt}, we get $D(\End(\IF))\cong
D(G \ltimes \End(\IE))$.

However, derived equivalence is a strictly weaker notion than Morita equivalence.\ In particular, there exist non-isomorphic but still derived equivalent basic algebras; see \cite{Happel--derivedfdalg},  \cite[Sect.\ 2.8 \& 2.9]{Keller--handbook}.

\subsection{Non-full exceptional sequences}
In \autoref{subsect:setup}, we make the general assumption that the exceptional sequence \eqref{eq:originalEC} that we start with is full. This is not really a restriction.

If we have an exceptional sequence satisfying all the assumptions
\autoref{subsect:setup} except that it is not full, we can replace $\cD$ by the thick, cocomplete triangulated subcategory $\langle\!\langle\IE \rangle\!\rangle\subset \cD$ generated by the exceptional sequence.

Then, inside $\langle\!\langle\IE \rangle\!\rangle$, the exceptional collection is full. Furthermore, $\langle\!\langle\IE \rangle\!\rangle$ is again dg-enhanced since $\langle\!\langle\IE \rangle\!\rangle\cong D(\End^\bullet(\IE))$; see the references in the proof of \autoref{DGequivalence}, or \cite[Thm.\ 1.10]{HK--P} for the exact formulation used here. Hence, we can apply our results to $\langle\!\langle\IE \rangle\!\rangle$ in place of $\cD$.

\subsection{Hilbert schemes of points on surfaces and symmetric quotient stacks} \label{subsect:symmquotstacks}

We come back to the example of \cite[Sect.\ 4]{KSos} mentioned in the introduction. Let $S$ be a smooth projective variety with an exceptional collection $(U_1, \dots ,U_\ell)$ in $D^b(S)\subset D(\QCoh(S))$.

Fix some $n\ge 2$, and let $X=S^n$ with the symmetric group $G=\sym_n$ acting on $X$ by permutation of the factors. Then one checks that the
objects
\[
 E(i_1,\dots, i_n):=E_{i_1}\boxtimes \dots \boxtimes E_{i_n}\in D^b(X)
\]
with $(i_1,\dots, i_n)$ going through the set $\{1,\dots,\ell\}^n$
with the lexicographic order form an exceptional sequence satisfying the assumptions of \autoref{thm:intro1} or, equivalently, the assumptions of \autoref{subsect:setup}. Hence, there is an induced exceptional sequence on the equivariant derived category $D^b_{\sym_n}(S^n)$ whose direct sum we denote by $\IF$. If $S$ is a surface, the derived McKay correspondence of Bridgeland--King--Reid \cite{BKR} and Haiman \cite{Hai} gives an equivalence $D^b_{\sym_n}(S^n)\cong D^b(S^{[n]})$, so we also get an induced exceptional sequence on the Hilbert scheme of points $S^{[n]}$.

Now write $\IU:=U_1\oplus\dots\oplus U_n$, $A:=\End(\IU)$ and
\[
 \IE:=\bigoplus_{(i_1,\dots, i_n)\in \{1,\dots,\ell\}^n} E(i_1,\dots,i_n)\,.
\]
Then we have $\IE\cong \IU^{\boxtimes n}$. Hence the K\"unneth formula yields $\End(\IE)\cong A^{\boxtimes n}$ with the induced $\sym_n$-action given by permutation of the tensor factors.\ By \autoref{thm:intro2}, we conclude that
\[
 \End(\IF)\cong \bigl(\sym_n\ltimes (A^{\otimes n})\bigr)^b = (\sym_n\wr A)^b
\]
is the basic reduction of the $n$-fold \emph{wreath product} algebra.

\subsection{The case of path algebras without relations}

Let $\fk=\bar\fk$ be algebraically closed.\ Then the path algebras of acyclic quivers without relations are exactly the finite-dimensional basic \emph{hereditary} algebras.
Let a finite group $G$ act on $A=\fk Q$ by algebra automorphisms for some acyclic quiver $Q=(Q_0,Q_1)$.\
By \cite[Thm.\ 1.1 \& 1.3(c)(i)]{RR--skew}, the associated skew group algebra $G\ltimes A$ is still hereditary.\ Therefore its basic reduction is again a path algebra of some acyclic quiver, say $Q_G$. 

Under the further assumption that the action permutes the trivial paths $e_i$ and preserves the subset $\fk\langle Q_1\rangle\subset \fk Q$ of linear combinations of arrows, the quiver $Q_G$ was computed in \cite[Thm.\ 1]{Demonet}.\ Our results recover this description as we explain in the following.\ Note however that this does not reprove the full result of \cite[Thm.\ 1]{Demonet}, since \emph{loc.\ cit}.\ is also true for cyclic quivers. 

Anyway, for our acyclic quiver $Q$, we have a full strong exceptional collection $(P(i))_{i\in Q_0}$ of $D(A)$ consisting of the indecomposable projectives $P(i):=e_iA$ associated to the vertices $i\in Q_0$.\ The vertices have to be ordered in such a way that $i\leq j$ whenever there is a path from $i$ to $j$ because $\Hom_A^{\ast}(P(i),P(j)) \cong e_jAe_i$.\ 
By assumption, there is a $G$-action on $Q_0$ such that $(e_i)^g=e_{ig}$ for all $g\in G$ and $i\in Q_0$. The action of every $g \in G$ on $A$ restricts to an isomorphism $g\colon P(i)\xrightarrow\sim g_*P(ig)$. In particular, every $P(i)$ admits a $\Stab_G(i)$-linearisation.

Hence, the full strong exceptional sequence $(P(i))_{i\in Q_0}$ satisfies the assumptions of \autoref{subsect:setup} with one block for every $G$-orbit of $Q_0$. \autoref{thm:fullEC} then gives a full exceptional collection $(F_{i,\rho})$ in $D(\fk Q)^G\cong D(G\ltimes \fk Q)$ where $i$ runs through a set of representatives of the $G$-orbits of $Q_0$ and $\rho$ runs through the irreducible representations of $\Stab_G(i)$.

By \autoref{rem:PhiKRS} and \autoref{thm:main}, the $F_{i,\rho}$ are in bijection with the indecomposable projectives in $(G\ltimes \fk Q)^b$ which in turn correspond to the vertices of the quiver $Q_G$ with $(G\ltimes \fk Q)^b\cong \fk Q_G$.
In summary, there is a bijection
\[
 (Q_G)_0\cong \coprod_{i} \irr(\Stab_G(i)),
\]
where $i$ runs through a set of representatives of the $G$-orbits of $Q_0$.\ This agrees with the description of the vertex set in \cite{Demonet}.\ With a little more work, the description of the arrows $(Q_G)_1$ in \emph{loc.\ cit.}\ can also be recovered by computing Hom spaces between the $F_{i,\rho}$.

\bibliographystyle{alphaurl}
\addcontentsline{toc}{chapter}{References}
\bibliography{references}

\end{document}